\documentclass[reqno]{amsart}%
\usepackage{amsmath}
\usepackage{amsfonts}
\usepackage{amssymb}
\usepackage{graphicx}%
\usepackage{hyperref}

\newtheorem{thm}{Theorem}[section]
\newtheorem{cor}[thm]{Corollary}
\newtheorem{lem}[thm]{Lemma}
\newtheorem{prop}[thm]{Proposition}

\theoremstyle{definition}
\newtheorem{defn}{Definition}[section]
\theoremstyle{Conjecture}

\theoremstyle{remark}
\newtheorem{rem}{Remark}[section]

\newcommand{\be}{\begin{equation}}
\newcommand{\ee}{\end{equation}}
\newcommand{\bea}{\begin{eqnarray}}
\newcommand{\eea}{\end{eqnarray}}
\newcommand{\ben}{\begin{eqnarray*}}
\newcommand{\een}{\end{eqnarray*}}
\newcommand{\bet}{\begin{equation}
\begin{split}}
\newcommand{\eet}{\end{split}
\end{equation}}

\begin{document}
\title[On some properties of squeezing functions of bounded domains]{On some properties of squeezing functions of bounded domains}
\date{}
\author{Fusheng Deng}
\address{F. Deng: School of Mathematical Sciences, Graduate University of
Chinese Academy of Sciences, Beijing 100049, China}
\email{fshdeng@gucas.ac.cn}
\author{Qian Guan}
\address{Q. Guan: Institute of Mathematics, Academy of Mathematics and Systems Science, Chinese Academy of Sciences, Beijing
100190, China}
\email{guanqian@amss.ac.cn}
\author{Liyou Zhang}
\address{L. Zhang: School of Mathematical Sciences, Capital Normal University, Beijing 100048, China}
\email{zhangly@mail.cnu.edu.cn}
\begin{abstract}
The main purpose of the present paper is to introduce the notion of squeezing functions of bounded domains and study some properties of them. The relation to geometric and analytic structures of bounded domains will be investigated. Existence of related extremal maps and continuity of squeezing functions are proved. Holomorphic homogeneous regular domains are exactly domains whose squeezing functions have positive lower bounds. Completeness of certain intrinsic metrics and pseudoconvexity of holomorphic homogeneous regular domains are proved by alternative method. In dimension one case, we get a neat description of boundary behavior of squeezing functions of finitely connected planar domains. This leads to a necessary and sufficient conditions for  a finitely connected planar domain to be a holomorphic homogeneous regular domain. Consequently, we can recover some important results in complex analysis. For annuli, we obtain some interesting properties of their squeezing functions. We finally exhibit some examples of bounded domains whose squeezing functions can be given explicitly.
\end{abstract}
\maketitle

\section{introduction}\label{sec:introduction}

Bounded domains are elementary objects of study in complex analysis. To study complex and geometric structures of bounded domains, one may consider holomorphic maps form bounded domains  to some standard domains such as balls and vice visa. The basic idea goes back to Carath\'{e}odory, and a typical example is the definition of Carath\'{e}odory metric and Kobayashi metric. Holomorphic maps from bounded  domains to the unit ball with certain extremal properties are called Carath\'{e}odory maps, which can be explicitly given for some special domains such as bounded symmetric domains  \cite{Kubota0}\cite{Kubota1} and ellipsoids \cite{Ma}.  Recently, by considering embeddings of general bounded domains into the unit ball,  a new concept of holomorphic homogeneous regular domains was introduced in \cite{Liu}. Holomorphic homogeneous regular domains are generalizations of Teichm\"uller spaces,  and they admit some nice geometric and analytic properties (see \cite{Liu}\cite{Yeung}).

Motivated by the works mentioned above, especially that in \cite{Liu}, we introduce the notion of squeezing functions defined on general bounded domains as follows:
\begin{defn}\label{def:squeezing function}
Let $D$ be a bounded domain in $\mathbb{C}^n$. For $p\in D$ and an (open) holomorphic embedding $f: D\rightarrow B^n$ with $f(p)=0$, we define
$$s_D(p , f)= \sup\{r|B^n(0,r)\subset f(D) \},$$
and the squeezing number $s_D(p)$ of $D$ at $p$ is defined as
$$s_D(p)= \sup_f\{s_D(p , f) \},$$
where the supremum is taken over all holomorphic embeddings $f: D\rightarrow B^n$ with $f(p)=0$, $B^n$ is the  unit ball in $\mathbb{C}^n$ and $B^n(0 , r)$ is the ball in $\mathbb{C}^n$ with center $0$ and radius $r$. We call $s_D$ the \emph{squeezing function} on $D$.
\end{defn}

By  definition, it is clear that  squeezing functions are invariant under biholomorphic transformations, so they can be viewed as a kind of holomorphic invariants of bounded domains. The main purpose of the present paper is to investigate some properties of squeezing functions and their relations with geometric and analytic structures of bounded domains.

Squeezing functions  are always positive and bounded above by 1. It is interesting to estimate their lower and upper bounds, which are numerical  holomorphic invariants of bounded domains, by the holomorphic invariance of squeezing functions.   Holomorphic homogeneous regular domains defined in \cite{Liu} are exactly bounded domains whose squeezing functions admit  positive lower bounds. They contain some interesting objects such as bounded homogeneous domains, Teichm\"uller spaces, bounded domains covering compact K\"{a}hler manifolds and strictly convex domains with $C^2$-boundary \cite{Yeung}.

It is easy to see that the squeezing function of the unit ball $B^n$ is constant with value 1. A natural question is whether the squeezing function of a bounded domain $D$ can attain the value 1 at some point $x$ in $D$ if $D$ is not holomorphic equivalent to $B^n$? This question is related to the existence of an extremal map  which realizes the supremum $s_D(x)$, i.e., the existence of a holomorphic embedding $f:D\rightarrow B^n$ such that $f(x)=0$ and $B^n(0, s_D(x))\subset f(D).$  We will prove the existence of extremal maps, by using a higher dimensional  generalization of  Hurwitz theorem that will be proved as well in the present article. As a consequence, $S_D$ attains the value 1 in $D$ if and only if $D$ is holomorphic equivalent to the unit ball. On the other hand, as we will see, there exist domains whose squeezing functions have supremum 1, but they are not holomorphic equivalent to the unit ball.

An elementary property of regularity of squeezing functions is their continuity. It can be proved by using the decreasing property of Kobayashi metrics. From the continuity property, one can see that a bounded domain is a holomorphic homogeneous regular domain if it covers a compact complex manifold.

Properties of squeezing functions can reflect some geometric and analytic properties of bounded domains. Boundary behavior of squeezing functions implies certain boundary estimate of Carath\'eodory metric, which implies completeness of the metric in some special cases. For a bounded domain whose squeezing function admits a positive lower bound, i.e.  a holomorphic homogeneous regular domain, it is known that the intrinsic metrics -- the Carath\'eodory metric, the Kobayashi metric and the Bergman metric -- on it are equivalent \cite{Liu}, and they are all complete \cite{Yeung}. We will prove the completeness of these metrics by alternative method based on Lu's result on comparsion of  Bergman metric and  Carath\'{e}odory metric \cite{Lu}. A result in several complex variables states that  completeness of the  Carath\'eodory metric of a domain implies its  pseudoconvexity (see e.g. \cite{Pflug}),  hence a holomorphic homogeneous regular domain must be a pseudoconvex domain.

Squeezing functions of planar domains have nice properties. For finitely connected planar domains, we get a neat description of the boundary behavior of their squeezing functions. As a result, we get the necessary and sufficient condition for such a domain to be a holomorphic homogeneous regular domain. Surprisingly, the squeezing function $s_D$ of a bounded planar domain $D$ with smooth boundary admits the boundary behavior

\begin{equation}\label{equ:boundary behavior}
\lim_{z\rightarrow \partial D}s_D(z)=1.
\end{equation}

By the continuity of squeezing functions, equality \eqref{equ:boundary behavior} implies that all smoothly bounded planar domains are holomorphic homogeneous regular domains. As a result, we can recover some important results about planar domains, e.g., the three intrinsic metrics mentioned above on a  bounded planar domain with smooth boundary are all complete, and they are equivalent, and a smoothly bounded planar domain must be hyperconvex, i.e., it admits a bounded exhaustive subharmonic  function.  In particular, equality \eqref{equ:boundary behavior} also implies that $s_D$ can be extended continuously to $\overline{D}$ for a planar domain $D$ with smooth boundary. We don't know whether this is true in general case, i.e.,  whether $s_D$ can be extended continuously to $\overline{D}$ for all bounded domains $D\subset \mathbb{C}^n$ with smooth boundary.

It is clear that the product of two holomorphic homogeneous regular domains is still a holomorphic homogeneous regular domain, so smoothly bounded  planar domains and their products provide a class of holomorphic homogeneous regular domains which are  generally not contained in the list of holomorphic homogeneous regular domains  mentioned above. As remarked in \cite{Liu}, it may be interesting to investigate whether the Kobayashi metric and the Carath\'{e}odory  metric on a Teichm\"uller space coincide or not; on the other hand, the Kobayashi metric and the Carath\'{e}odory  metric on general holomorphic homogeneous regular domains constructed here don't coincide.

The simplest nontrivial  smoothly bounded planar domains are annuli. However, even in this special case, squeezing functions admit nontrivial  properties. With certain investigation, we conjecture that the conformal structure of an annulus is characterized by the exact lower bound of its squeezing function.

The squeezing functions can be given explicitly for classical bounded symmetric domains. In this special case, we see that the extremal maps for squeezing functions defined as above can be given by the Carath\'{e}odory maps (see \S\ref{sec:explicit form of sf} for exact definition). However, this does not hold for general domains such as annuli. In fact, the Carath\'{e}odory maps (often called Ahlfors maps  for planar domains) of a bounded planar domain can not even be injective if the domain is not simply connected (see e.g. \cite{Fisher}). It seems that the obstruction for the coincidence of the two types of extremal maps comes from topology. Therefore, we conjecture that the extremal maps of a contractible domain are given by Carath\'{e}odory maps.

The rest of this article is organized as follows. In \S\ref{sec:existence of extremal function}, we generalize the Hurwitz theorem from one complex variable analysis to several complex variables,  and use this generalization to establish the existence of extremal maps that are defined as above; In \S\ref{sec:continuity}, we prove the continuity of squeezing functions of general bounded domains; In \S\ref{sec:relations with metrics}, we give a boundary estimate of Carath\'eodory metrics in term of boundary behavior of squeezing functions, and prove the completeness of certain intrinsic metrics on holomorphic homogeneous regular domains , as a corollary, we get the pseudoconvexity of these domains; In \S\ref{sec:planar domain case}, we  study squeezing functions on planar domains, and prove equality \eqref{equ:boundary behavior} of smoothly bounded planar domains, we also construct a class of planar holomorphic homogeneous regular domains which are infinitely connected;  In  \S\ref{sec:annuli}, we focus on squeezing functions on annuli, and in the final \S\ref{sec:explicit form of sf}, we give some examples of bounded domains whose squeezing functions can be given explicitly.

\vspace{.1in} {\em Acknowledgements}. The authors would like to thank Xiangyu Zhou,
the PhD advisor of the first two authors, for invaluable instructions and discussions. They are also grateful to Boyong Chen, Kefeng Liu, Peter Pflug, and Sai-Kee Yeung for helpful discussions. The authors are partially supported by NSFC grants (10901152 and 11001148).

\section{Generalized Hurwitz Theorem and the existence of extremal functions}\label{sec:existence of extremal function}

The main aim of this section is to establish the existence of extremal maps related to squeezing functions, i.e., the following

\begin{thm}\label{thm:existence of ext}
Let $D$ be a bounded domain in $\mathbb{C}^n$, then for any $x\in D$, there exists a  holomorphic embedding $f: D\rightarrow B^n$ such that
$f(x)=0$ and $B^n(0, s_D(x))\subset f(D)$.
\end{thm}

By definition we have $s_D(z)\leq 1$ for all $z\in D$.
By Theorem \ref{thm:existence of ext}, we see that $s_D(z)=1$ for some $z\in D$ if and only if $D$ is holomorphically equivalent to $B^n$.

To prove Theorem \ref{thm:existence of ext}, we need to generalize Hurwitz's theorem  in classical complex analysis to several complex variables. Hurwitz's theorem in one complex variable says that the limit of a sequence of univalent functions on a planar domain is univalent  unless it is constant (see e.g. \cite{Remmert}). Of course there is no direct generalization of this result in higher dimensions, however, a modified version described in the following theorem still holds:

\begin{thm}\label{thm:generalized Hurwitz}
Let $D$ be a bounded domain in $\mathbb{C}^n$ and $x\in D$, let $f_i$ be a sequence of injective holomorphic maps from $D$ to  $\mathbb{C}^n$
such that $f_i(x)=0\in \mathbb{C}^n$ for all $i$. Suppose $f_i$ converges  to a map $f:D\rightarrow \mathbb{C}^n$ uniformly on compact subsets of $D$. If there exists a neighborhood $U$ of $0$ in $\mathbb{C}^n$ such that $U\subset f_i(D)$ for all $i$, then $f$ is injective.
\end{thm}

 To prove Theorem \ref{thm:generalized Hurwitz}, we need two lemmas. The first is
\begin{lem}\label{lem: stability of immersion}Let $D$ be a domain in $\mathbb{C}^n$ and $\varphi_i$ a sequence of holomorphic maps from $D$ to $\mathbb{C}^n$ which is convergent to a map $\varphi:D\rightarrow \mathbb{C}^n$ uniformly on compact subsets of $D$. If all $\varphi_i$ have no zero in $D$, then $\varphi$ has no zero in $D$ unless it is identically zero.
\end{lem}
\begin{proof}
By the identity theorem of holomorphic functions, we may assume $D$ is a ball. Assume there exists $x\in D$ such that $\varphi(x)=0$. For any point $z\in D$, consider the intersection of $D$ and the complex line containing $x$ and $z$, then a version of the classical Hurwitz theorem (see Corollary in P. 162 in \cite{Remmert} ) implies that $\varphi(z)=0$. So $\varphi$ is identically
zero on $D$.
\end{proof}

The second lemma we need in the proof of Theorem \ref{thm:generalized Hurwitz} is the generalized Rouch\'{e}'s theorem in higher dimensions, whose proof  relies on the mapping degree theory in differential topology.
\begin{lem}\label{lem:generalized Rouche}
(Theorem 3 in \cite{Lloyd}) Let $D$ be a bounded domain in $\mathbb{C}^n$, suppose $f$ and $g$ are two holomorphic maps from $D$ to $\mathbb{C}^n$
such that
$$\parallel g(z)\parallel<\parallel f(z)\parallel,\ \ z\in \partial D$$
Then $f$ and $f + g$ have the same number of zeros in D, counting multiplicities, where $\parallel\cdot \parallel$ is the standard norm in $\mathbb{C}^n$.
\end{lem}

 With the above two lemmas, we can give the proof of Theorem \ref{thm:generalized Hurwitz} as follows:
 \begin{proof}(Proof of Theorem \ref{thm:generalized Hurwitz})Let $g_i = f_i^{-1}|_U$. By Montel theorem, the sequence $\{g_i\}$ is  convergent to a holomorphic map $g:U\rightarrow \mathbb{C}^n$ uniformly on compact subsets of $U$. Note that $g(0) = x$ is a interior point of $D$,  we can assume $g(U)\subset D$ by taking $U$ smaller enough. It is clear that $f_ig_i=Id_U$ for all $i$. Let $i$ tends to $\infty$, we get $fg = Id_U$. This implies that the Jacobian determinant $detJ_f(x)$ of $f$ at $x$ is not zero. Since $f_i$ are injective, $detJ_{f_i}(z)\neq 0$ for all $i$ and all $z\in D$ (see e.g. Theorem 8.5 in \cite{Fritzsche}). Note that $detJ_{f_i}$ converges to $detJ_f$ uniformly on compact subsets of $D$, by Lemma \ref{lem: stability of immersion}, we see that $detJ_f(z)\neq 0$ for all $z\in D$.
 We  prove $f$ is injective. If it is not the case, there exist $z_1 , z_2\in D$, $z_1\neq z_2$, such that $f(z_1)=f(z_2)$. Since $detJ_f\neq 0$, we can choose a neighborhood $\Omega\subset\subset D$ of $z_1$  such that $f|_{\overline{\Omega}}$ is injective and $z_2\not\in\overline{\Omega}$.
 Set $\tilde{f_i}=f_i-f_i(z_1)$ and  $\tilde{f}=f-f(z_1)$, then $\tilde{f_i}$ converges to $\tilde{f}$ uniformly on $\overline{\Omega}$. Note that $\tilde{f}$ has a zero in $\Omega$ and $\tilde{f}^{-1}(0)\cap \partial\Omega = \emptyset$, by Lemma \ref{lem:generalized Rouche}, $\tilde{f_i}$ has a zero in $\Omega$ for $i$ large enough, which  contradicts to the fact that $f_i$ are all injective on $D$.
 \end{proof}

\begin{rem}
 The assumptions in Theorem \ref{thm:generalized Hurwitz} that $D$ is bounded and all $f_i(D)$ contain a fixed neighborhood of $0\in \mathbb{C}^n$ are necessary. In fact, the result in Theorem \ref{thm:generalized Hurwitz} does not hold without any one of the assumptions. For example, taking
$f_i(z_1,z_2)= (z_1, z_2/i)$
as a sequence  holomorphic maps from $D=\mathbb{C}^2$ to itself, the limit map $f:D\rightarrow\mathbb{C}^2$ is given by $f(z_1, z_2)=(z_1 , 0)$, it is not injective even  $f_i$ are injective and $\mathbb{C}^2\subset f_i(D)$ for all $i$; the restrictions $f_i|_{B^2}$ give a sequence of injective holomorphic maps from $B^2$ to $\mathbb{C}^2$ with limit map  $f|_{B^2}$, which is not injective since not all $f_i(B^2)$ contain a fixed neighborhood of $0\in \mathbb{C}^2$. On the other hand, by a result  in \cite{Hahn}, the two assumptions can be replaced by assuming  that $|detJ_{f_i}(x)|$ have a positive lower bound.
\end{rem}
Using the results in Theorem \ref{thm:generalized Hurwitz} and Lemma \ref{lem: stability of immersion}, we give the proof  of Theorem \ref{thm:existence of ext}:
\begin{proof}(Proof of Theorem \ref{thm:existence of ext})
By definition of squeezing functions, there exist a sequence of holomorphic embeddings $f_i:D\rightarrow B^n$ with $f_i(x)=0$, and a sequence of increasing positive numbers $r_i$ convergent to $s_D(x)$ such that $B^n(0 , r_i)\subset f_i(D)$. By Montel theorem, there exists a subsequence $\{f_{i_k}\}$ of $f_i$ which converges to a homomorphic map $f:D\rightarrow \mathbb{C}^n$ uniformly on compact subsets of $D$. Since $B^n(0 , r_1)\subset f_i(D)$ for all $i$, by Theorem \ref{thm:generalized Hurwitz}, $f$ is injective. In particular it is an open map and hence $f(D)\subset B^n$. Then we get a holomorphic embedding $f:D\rightarrow B^n$ with $f(0)=0$.\\
To prove $B^n(0 , s_D(x))\subset f(D)$, it suffices to prove $B^n(0 , r_{j})\subset f(D)$ for each fixed integer $j$. By assumption, $B^n(0 , r_{j})\subset f_i(D)$ for all $i>j$. Let $g_i = f^{-1}_i|_{B^n(0 , r_{j})}$, then we have $f_{i_k}g_{i_k} = Id_{B^n(0 , r_{j})}$ for $i_k>j$. By Montel theorem,  we may assume that the sequence $\{g_{i_k}\}$ converges to a holomorphic map $g:B^n(0 , r_{j})\rightarrow \mathbb{C}^n$ uniformly on compact subsets of $B^n(0 , r_{j})$. We want to prove that $g(B^n(0 , r_{j}))\subset D$. Note that $g(0)=x$, hence there exists a neighborhood $U$ of $0$ in $B^n(0 , r_{j})$ such that $g(U)\subset D$. This implies $f\cdot g|_U$ is defined and it is clearly equal to the identity map $Id_U$. So $detJ_g(0)\neq 0$. Since $detJ_{g_i}\neq 0$ for all $i>i_0$, by Lemma \ref{lem: stability of immersion}, we have $det J_f\neq 0$ and hence $g$ is an open map, which implies that $g(B^n(0 , r_{j}))\subset D$. Therefore  $fg:B^n(0 , r_{j})\rightarrow B^n(0 , r_{j})$ is a well defined map. It is clear that $fg=Id_{B^n(0 , r_{j})},$  so we have $B^n(0 , r_{j})\subset f(D)$.\\

\end{proof}

\section{Continuity of squeezing functions}\label{sec:continuity}
In this section, we will prove that the squeezing function on any bounded domain is  continuous. As a consequence, a bounded domain is a holomorphic homogeneous regular domain if it  covers a compact complex manifold.

\begin{thm}\label{thm:continuity of squeezing functions}
The squeezing function $s_D$ of any bounded domain $D$ in $\mathbb{C}^n$ is continuous.
\begin{proof}
Since $D$ is a bounded domain,  the Kobayashi metric on $D$ is nondegenerate.

Let $a$ be an arbitrary point in $D$ and $\{z_k\}$ be a sequence in $D$ convergent to $a$, and let $\epsilon$ be an arbitrary positive number. By Theorem \ref{thm:existence of ext},
there exists a holomorphic embedding $f: D\rightarrow B^n$ such that $f(a)=0$ and $B^n(0,s_D(a))\subset f(D)$. Since $f$ is continuous, there exists an integer  $N$ such that $||f(z_k)||<\epsilon$ for $k>N$. Define $f_k:D\rightarrow \mathbb{C}^n$ as
$$f_k(z) =  \frac{f(z)-f(z_k)}{1+\epsilon}$$
for $k>N$, then $f_k(D)\subset B^n$, $f_k(z_k)=0$ and $$B^n(0, \frac{s_D(a)-\epsilon}{1+\epsilon})\subset f_k(D).$$ This implies that $s_D(z_k)\geq (s_D(a)-\epsilon)/(1+\epsilon)$. Let $\epsilon$ tends to 0, we get
$$\liminf_{k\rightarrow \infty} s_D(z_k)\geq s_D(a).$$
Let $K_D(\cdot , \cdot)$ be the Kobayashi distance on $D$. It is known that  $K_D$ is continuous on $D\times D$ (see e.g. \cite{Kobayashi}). So we have $K_D(z_k,a)\rightarrow 0$ as $k\rightarrow \infty$. By Theorem \ref{thm:existence of ext}, for each $k$, there exists a holomorphic embedding $f_k: D\rightarrow B^n$ such that $f_k(z_k)=0$ and $B^n(0, s_D(z_k))\subset f_k(D)$. By the decreasing property of  Kobayashi distances (see e.g. \cite{Kobayashi}), we have $$K_{B^n}(f_k(z_k), f_k(a)) = K_{B^n}(0, f_k(a))\leq K_D(z_k,a)$$
for all $k$. So $K_{B^n}(0, f_k(a))\rightarrow 0$, which implies that $f_k(a)$ tends to 0 in the ordinary topology (see \cite{Barth}). So, for any positive number $\epsilon$, there exists an integer $M$ such that $||f_k(a)||< \epsilon$ for $k>M$. This implies
$$s_D(a)\geq \frac{s_D(z_k)-\epsilon}{1+\epsilon},$$
so we have
$$s_D(a)\geq\limsup_{k\rightarrow \infty}\frac{s_D(z_k)-\epsilon}{1+\epsilon}.$$
Letting $\epsilon$ tends to 0, we get
$$s_D(a)\geq\limsup_{k\rightarrow \infty}s_D(z_k).$$
So $\lim_{k\rightarrow \infty}s_D(z_k)= s_D(a)$, namely $s_D$ is continuous at $a$. Note that $a$ is arbitrary, so $s_D$ is continuous on $D$.
\end{proof}
\end{thm}

For  $r\in [0,1)$, we define
$$\sigma(r)= log\frac{1+r}{1-r}.$$
It is clear that $\sigma(c)$ is strictly increasing for $0\leq  c<1$, its inverse is given by $\sigma^{-1}(w)= \tanh(w/2)$.
For a point $z\in B^n$,  the Kobayashi distance form $0$ to $z$ is $\sigma(|z|)$. For two nonnegative numbers $u$ and $v$, it is not difficult to prove that
$$\sigma^{-1}(u+v)\leq \sigma^{-1}(u)+\sigma^{-1}(v).$$
Let $D$ be a bounded domain in $\mathbb{C}^n$, we define a function $T(\cdot , \cdot)$ on $D\times D$ as
$$T(x,y)= \sigma^{-1}(K_D(x,y))$$
then the above properties of $\sigma$ implies that $T(\cdot , \cdot)$ is a metric on $D$. Since $K_D$ induces the ordinary topology of $D$, so does $T$. From the proof of Theorem \ref{thm:continuity of squeezing functions}, one can directly get the following
\begin{thm}
The squeezing function $s_D$ of $D$ is Lipschitz continuous with respect to the metric $T$. In fact, we have
$$|s_D(x)- s_D(y)|\leq 2T(x , y), \ x, y \in D.$$
\end{thm}
\begin{rem}
The same result as in the above theorem still holds if we replace Kobayashi distance in the definition of $T(\cdot , \cdot)$ by Carath\'eodory distance.
\end{rem}

By Theorem \ref{thm:continuity of squeezing functions}, we directly get the following result proved in \cite{Yeung}:
\begin{cor}
Let $D$ be a bounded domain that  covers a compact complex manifold, then $D$ is a holomorphic homogeneous regular domain.
\begin{proof}
Let $X$ be a compact complex manifold that is covered by $D$. By the holomorphic invariance of squeezing functions, $s_D(z)$ can be pushed down to a function on $X$. By Theorem \ref{thm:continuity of squeezing functions}, $s_D(z)$ is continuous. Note that $s_D(z)$ is also positive, it must attain a positive lower bound on $X$, and hence on $D$.
\end{proof}
\end{cor}

\section{Relations between intrinsic metrics and squeezing functions}\label{sec:relations with metrics}

The main purpose of this section is to investigate relations between squeezing functions and some intrinsic metrics on bounded domains.
We give a boundary estimate of the Carath\'eodory metric of a bounded planar domain in term of boundary behavior of its squeezing function.
In fact,  similar but weaker form of this result still holds in higher dimensional case.  We then focus on bounded holomorphic homogeneous regular domains, and prove that the Carath\'{e}odory metric, the Kobayashi metric and the Bergman metric  on these domains are complete. As a result, a holomorphic homogeneous regular domain must be pseudoconvex.

We  first need the following lemma, which is known as Koebe's one-quarter theorem in classical complex analysis.
\begin{lem}\label{lem:Koebe 1/4}(see e.g. \cite{Ahlfors})
Let $\Delta\subset \mathbb{C}$ be the unit disc. Let $g$ be a univalent holomorphic function on $\Delta$ such that $g(0)=0$ and $g'(0)=1$. Then
$\Delta_{1/4}:=\{z\in\mathbb{C}; |z|<1/4\}\subset g(\Delta)$.
\end{lem}

With this lemma, we now prove the following
\begin{thm}\label{thm:squeezing functions vs Caratheodory}
Let $D$ be a bounded domain in $\mathbb{C}$, $x\in D$. Then the Carath\'{e}odory pseudo-norm of $\frac{\partial}{\partial z}$ at $x$ is not less than $s_D(x)/4\delta(x)$, where $z$ is the standard coordinate on $\mathbb{C}$ and $\frac{\partial}{\partial z}$ is viewed as a vector in the tangent space $T_xD$ of $D$ at $x$.
\end{thm}
\begin{proof}
 By Theorem \ref{thm:existence of ext}, there exists a univalent map $f:D\rightarrow \Delta$ such that $f(x)=0$ and $\Delta_{s_D(x)}\subset f(D)$, where $\Delta_{s_D(x)}$ is the disc in $\mathbb{C}$ with center $0$ and radius $s_D(x)$. We want to estimate the module  $|f'(x)|$ of the derivative of $f$ at $x$. Let $$g=f^{-1}|_{\Delta_{s_D(x)}},$$
it is a univalent map form $\Delta_{s_D(x)}$ to $D$ such that $g(0)=x$.

Now we define a univalent map $\varphi:\Delta\rightarrow \mathbb{C}$ by setting
$$\varphi(z)= \frac{g(s_D(x)\cdot z)-x}{s_D(x)\cdot g'(0)},$$
then it is clear that $\varphi(0)=0$ and $\varphi'(0)=1$. By Lemma \ref{lem:Koebe 1/4}, we have $\Delta_{1/4}\subset \varphi(\Delta)$. This implies that
$$\Delta(x, s_D(x)|g'(0)|/4)\subset g(\Delta_{s_D(x)})\subset D,$$
where, for $a\in \mathbb{C}$ and $r>0$, we set $\Delta(a , r)$  the disc in $\mathbb{C}$ with center $a$ and radius $r$.
In particular, we have
$$\delta(x)\geq s_D(x)|g'(0)|/4.$$
Note that $f'(x)= 1/g'(0)$, we get
$$|f'(x)|\geq \frac{s_D(x)}{4\delta(x)}.$$
This means that the Carath\'{e}odory pseudo-norm of $\frac{\partial}{\partial z}$ at $x$ is not less than $s_D(x)/4\delta(x)$.
\end{proof}

\begin{rem}
Using similar argument as in the proof of Theorem \ref{thm:squeezing functions vs Caratheodory}, one can prove a weaker form of Theorem \ref{thm:squeezing functions vs Caratheodory} in higher dimensional cases. In fact, for a bounded domain $D\subset \mathbb{C}^n$,  one can prove that the Carath\'{e}odory pseudo-norm $\parallel X\parallel_{C_D}$ on $D$ of $X\in T_xD=\mathbb{C}^n$ admits the estimate
$$\parallel X\parallel_{C_D}\geq \frac{s_D(x)\parallel X\parallel}{4\delta(x,X)},$$
where $\delta(x,X)$ is the boundary distance of $x$ with respect to the direction $X$, and $\parallel X\parallel$ is the Euclidean norm of $X$.
\end{rem}

A corollary of theorem\ref{thm:squeezing functions vs Caratheodory} is the following
\begin{thm}\label{thm:boundary estimate}
Let $D$ be a bounded domain in $\mathbb{C}$ satisfying
$$s_D(x) > C/log(1/\delta(x))$$
for some positive constant $C$ and all $x\in D$ with $\delta(x)<1$, then the Carath\'{e}odory metric on $D$ is complete.
\end{thm}
\begin{proof}
By Theorem \ref{thm:squeezing functions vs Caratheodory}, we see that the Carath\'{e}odory pseudo-norm of $\frac{\partial}{\partial z}$ at $x$ is not less than $\frac{-C}{\delta(x)log\delta(x)}$, which  implies the completeness of the Carath\'{e}dory metric on $D$.
\end{proof}

We focus on holomorphic homogeneous regular domains in the rest of this section. We first prove the completeness of the Carath\'{e}odory metric of a holomorphic homogeneous regular domain.
\begin{thm}\label{thm:Caro. metic on holomorphic homogeneous regular domain complete}
Let $D$ be a holomorphic homogeneous regular domain in $\mathbb{C}^n$ . Then the Carath\'{e}odory metric on $D$ is complete.
\begin{proof}
Since  $D$ is a bounded domain, the Carath\'{e}odory metric on $D$ is nondegenerate. Denote by $C_D(\cdot , \cdot)$ the Carath\'{e}odory distance on $D$.

Let $c>0$ be a positive lower bound of the squeezing function $s_D$ of $D$. We first  prove that, for any $x_0\in D$,  the set
$$A(x_0):= \{x\in D; C_D(x,x_0)< \log(\frac{1+c/2}{1-c/2}) \}$$
is relatively compact in $D$. By theorem\ref{thm:existence of ext}, there exists an open holomorphic embedding $f_{x_0}:D\rightarrow B^n$ such that $f_{x_0}(x_0)=0$ and $B^n(0,c)\subset f_{x_0}(D)$. By the decreasing property of  Carath\'{e}odory metrics, we have
$$f_{x_0}^*C_{B^n}\leq C_D$$
this implies that
$$f_{x_0}(A(x_0))\subset \{z\in B^n; C_{B^n}(z,0)< \log(\frac{1+c/2}{1-c/2})\}.$$
Note that the set
$$\{z\in B^n; C_{B^n}(z,0)< \log(\frac{1+c/2}{1-c/2})\}=\{z\in B^n; ||z||<c/2\}$$
 is relatively compact in $B^n(0 , c)$, so $A(x_0)$ is relatively compact in $D$.

For bounded domains, it is known that the topology induced by the Carath\'{e}odory metric coincides with the ordinary topology(see e.g. \cite{Pflug}). In particular, a compact set of $D$ with the ordinary topology is also compact with respect to the topology induced by the Carath\'{e}odory metric. By the above result, we see that
the Carath\'{e}odory metric on $D$ is complete.
\end{proof}
\end{thm}

It is known that a domain whose Carath\'{e}odory metric is complete must be pseudoconvex (see e.g. \cite{Pflug}), as a consequence of Theorem \ref{thm:Caro. metic on holomorphic homogeneous regular domain complete}, we have the following result, which was proved in \cite{Yeung} by different method:

\begin{cor}\label{cor:holomorphic homogeneous regular domain implies stein}
A holomorphic homogeneous regular domain must be pseudoconvex.
\end{cor}

For any complex manifold, it is well known that its Carath\'{e}odory pseudo-metric is always dominated by its Kobayashi pseudo-metric (see e.g. \cite{Kobayashi}). For bounded domains, a famous result of Lu in \cite{Lu} says that the Carath\'{e}odory metric is always dominated by the Bergman metric. Note also that, for a bounded domain $D$, any one of the three intrinsic metrics- the Carath\'{e}odory metric, the Kobayashi metric, and the Bergman metric- induces  the same topology as the ordinary one.
 Therefore,  as a consequence of Theorem \ref{thm:Caro. metic on holomorphic homogeneous regular domain complete}, we have the following

\begin{thm}\label{thm:intrinsic metric on holomorphic homogeneous regular domain complete}
Let $D$ be a holomorphic homogeneous regular domain. Then the Kobayashi metric and the Bergman metric on it are complete.
\end{thm}

\section{Squeezing functions on planar domains }\label{sec:planar domain case}
In this section, we will consider squeezing functions on planar domains. For finitely connected planar domains, we get a neat description of the boundary behavior of their squeezing functions. As a result, we get the necessary and sufficient condition for such a domain to be a holomorphic homogeneous regular domain. If $D$ has smooth boundary, then
$$\lim_{z\rightarrow \partial D} s_D(z) = 1.$$
By continuity of $s_D$, this implies $D$ is a holomorphic homogeneous regular domain. As a consequence, we can recover some important results about planar domains, for example, the three intrinsic metrics-the Carath\'eodory metric, the Kobayashi metric, and the Bergman metric-on a bounded planar domain with smooth boundary are all complete, and they are equivalent;  all smoothly bounded planar domains  are hyperconvex , i.e., they admit bounded exhaustive subharmonic functions.  We also give a class of holomorphic homogeneous regular domains  which are infinitely connected.

It is clear that the squeezing function of the unit disc is the constant function with value 1. By Riemann mapping theorem and  holomorphic invariance of squeezing functions, the squeezing function of any simply connected bounded planar domain is also constant with value 1.

Now we consider squeezing functions on 2-connected planar domains. Define
$$A_r = \{z\in\mathbb{C}; r<|z|<1\}$$
for $0\leq r <1$. When $r>0$, we call $A_r$ an annulus. It is well-known that a 2-connected domain in $\mathbb{C}$ which is not conformal equivalent to $\mathbb{C}^*$  must be holomorphic equivalent to a unique $A_r$ (see e.g. \cite{Ahlfors1}). When $r=0$, $A_0$ is just the punctured disc $\Delta^*$, and we will consider it in the last section.
For $A_r$ with $r>0$, we have the following

\begin{thm}\label{thm:squeezing function of annuli}
For $r>0$, the squeezing function $s_{A_r}(z)$ tends to 1 as $z\rightarrow \partial A_r$. In particular,  $A_r$ is a holomorphic homogeneous regular domain.
\end{thm}
\begin{proof}
For  $c\in [0,1)$, we define
$$\sigma(c)= log\frac{1+c}{1-c}.$$
It is clear that $\sigma(c)$ is strictly increasing for $0\leq  c<1$, its inverse is given by $\sigma^{-1}(w)= \tanh(w/2)$.
For a point $z\in \Delta$,  the Poincar\'e distance form $0$ to $z$ is $\sigma(|z|)$.

Now let $z\in A_r$, with respect to the Poincar\'e metric on $\Delta$, the distance from $z$ to the cycle $\{w; |w|=r\}$ is $\sigma(|z|)-\sigma(r)$. Denote by $P(z, \sigma(|z|)-\sigma(r))$ the disc
(w.r.t the Poincar\'e metric on $\Delta$) with center $z$ and radius $\sigma(|z|)-\sigma(r)$, then we have $P(z, \sigma(|z|)-\sigma(r))\subset A_r$. Choose a conformal map $f: \Delta\rightarrow \Delta$ such that $f(z)=0$. Since $f$ preserves the Poincar\'e metric on $\Delta$, it maps $P(z, \sigma(|z|)-\sigma(r))$  onto the disc (w.r.t the Poincar\'e metric on $\Delta$) with center $f(z)=0$ and radius $\sigma(|z|)-\sigma(r)$, which is just the Euclidean disc with center 0 and radius $\sigma^{-1}(\sigma(|z|)-\sigma(r))$. This implies that
\be\label{squeezing function of annuli}
s_{A_r}(z)\geq \sigma^{-1}(\sigma(|z|)-\sigma(r)).
\ee
Note that $\sigma^{-1}(\sigma(|z|)-\sigma(r))\rightarrow 1$ as $|z|\rightarrow 1$, so  $s_{A_r}(z)$ tends to 1 as $|z|\rightarrow 1$.
Consider the holomorphic automorphism of $A_r$ given by $z\mapsto r/z$,  by the conformal invariance of $s_{A_r}$, we also get $s_{A_r}(z)$ tends to 1 as $|z|\rightarrow r$.

By Theorem \ref{thm:continuity of squeezing functions},  $s_{A_r}$ is continuous and hence has a positive lower bound, so $A_r$ is a holomorphic homogeneous regular domain.
\end{proof}

By  similar argument, Theorem \ref{thm:squeezing function of annuli} can be generalized to finitely connected planar domains as follows:

\begin{thm}\label{thm:fin. connected not point}
Let $D$ be a  domain in $\mathbb{C}$, assume that $\bar{\mathbb{C}}-D$ has finitely many connected components such that each connected component is not a single point. Then we have
$$\lim_{z\rightarrow \partial D} s_D(z) = 1.$$
In particular, $D$ is a holomorphic homogeneous regular domain.
\end{thm}
\begin{proof}
We define the function $\sigma$ as in the proof of  Theorem \ref{thm:squeezing function of annuli}. Let $E_1, \cdots , E_n$ be  connected components of $\bar{\mathbb{C}}-D$ . Then $D_1:=\bar{\mathbb{C}}-E_1$ is simply connected and $D\subset D_1$. Since $E_1$ is not a single point, by  Riemann mapping theorem, there is a conformal map $\varphi_1$ from $D_1$ to $\Delta$. It is clear that $\varphi_1(D)$ is a domain in $\Delta$ obtained by deleting $n-1$ connected compact subsets, say $L_2, \cdots , L_n$, from $\Delta$.

Let $z \in D $ and let $l_z = P_\Delta(z , \cup_{i=2}^nL_i)$  be the distance from $z$ to $\cup_{i=2}^nL_i$ with respect to the Poincare distance of $\Delta$. Choose a conformal map $f: \Delta\rightarrow \Delta$ such that $f(z)=0$, then the Euclidean disc with center $0$ and radius $\sigma^{-1}(l_z)$ is contained in $f(\varphi(D))$, which implies that $s_{\varphi(D)}(z)\geq \sigma^{-1}(l_z)$. If $|z|$ tends to 1, then $l_z$ tends to $\infty$ and $\sigma^{-1}(l_z)$ tends to 1, hence $s_{\varphi(D)}(z)$ tends to 1.
By  holomorphic invariance of squeezing functions, we see that
$$\lim_{z\rightarrow E_1} s_D(z) = 1.$$
Similarly, for $E_i$ with $2\leq i\leq n$, we also have
 $$\lim_{z\rightarrow E_i} s_D(z) = 1.$$
Hence
$$\lim_{z\rightarrow \partial D} s_D(z) = 1.$$

By continuity of $s_D$,  it admits a positive lower bound on $D$, so $D$ is a holomorphic homogeneous regular domain.
\end{proof}

Using Riemann mapping theorem, one can prove that the domains considered in the above theorem are holomorphic equivalent to bounded domains with smooth boundary(see e.g. \cite{Ahlfors1}). Hence an equivalent version of the above theorem is the following:

\begin{thm}\label{thm:smooth planar domain}
Let $D$ be a bounded domain in $\mathbb{C}$ with smooth boundary. Then we have
$$\lim_{z\rightarrow \partial D} s_D(z) = 1$$
In particular, $D$ is a holomorphic homogeneous regular domain.
\end{thm}
\begin{rem}
As a consequence of Theorem \ref{thm:smooth planar domain},  for a bounded planar domain $D$ with smooth boundary, $s_D$ can be extended continuously to $\bar{D}$. It may be interesting to investigate whether the same result holds in higher dimensions.
\end{rem}

In section 3, we have shown that the Carath\'{e}odory metric, the Kobayashi metric, and the Bergman metric on a holomorphic homogeneous regular domain are complete. As mentioned in the introduction, these intrinsic metrics on a holomorphic homogeneous regular domain are equivalent \cite{Liu}. It is also turned out  that any holomorphic homogeneous regular domain is hyperconvex \cite{Yeung}.  So, as a result of  Theorem \ref{thm:smooth planar domain}, we can recover the following results in complex analysis:

\begin{thm}
Let $D$ be a planar domain with smooth boundary, then we have:\\
(1) The Carath\'{e}odory metric, the Kobayashi metric and the Bergman metric on $D$ are complete;\\
(2) The Carath\'{e}odory metric , the Kobayashi metric and the Bergman metric  on D are equivalent.\\
(3) $D$ is hyperconvex.
\end{thm}

By definition, it is clear that the product of two holomorphic homogeneous regular domains is again a holomorphic homogeneous regular domain. So we get
\begin{cor}\label{cor:product of planar domains}
Let $D$ be a bounded domain in $\mathbb{C}^n$ which is holomorphic equivalent to the product of bounded planar domains with smooth boundary, then $D$ is a holomorphic homogeneous regular domain.
\end{cor}

As mentioned in the introduction, the list of known holomorphic homogeneous regular domains  contains bounded homogeneous domains , Teichm\"uller spaces, bounded domains covering compact K\"ahler manifolds and strictly convex domains with $C^2$-boundary.
Many examples of holomorphic homogeneous regular domains given by  Corollary \ref{cor:product of planar domains} are not in the list . More precisely, we have the following

\begin{prop}
Let $D_1, \cdots , D_k$ be bounded planar domains with smooth boundaries which are mutually not conformal equivalent. If there exists a $D_i$ which is not conformal equivalent to the unit disc $\Delta$, then the domain $D:= D_1^{r_1}\times\cdots\times D_k^{r_k}$ is a holomorphic homogeneous regular domain which  is not holomorphic equivalent to any of the domains in the above list. Where $r_1, \cdots , r_k$ are positive integers and $D_i^{r_i}= D_i\times\cdots\times D_i$ is the $r_i$-power of $D_i$.
\begin{proof}
Denote by $Aut(D)$ the holomorphic automorphism group of $D$. By the proposition in \cite{Royden} and Theorem 1 in \cite{Urata}, we have
$$Aut(D) = Aut(D_1^{r_1})\times\cdots\times Aut(D_k^{r_k}),$$
and, for each $i$, the following sequence is exact:
$$1\rightarrow (Aut(D_i))^{r_i}\rightarrow Aut(D_i^{r_i})\rightarrow S_{r_i}\rightarrow 1,$$
where $S_{r_i}$ is the symmetry group of degree $r_i$ which acts on $D_i^{r_i}$   by permutation.
The decomposition of $Aut(D)$ implies that  $D$ is homogeneous if and only if each factor $D_i$ is homogeneous, and $D$ can cover a compact complex manifold is and only if so does for each $D_i$. Note that a smoothly bounded planar domain is homogeneous or can cover a compact complex manifold if and only if it is isomorphic to $\Delta$, hence $D$ can not be homogeneous or cover a compact complex manifold.

It is clear that $D$ can not be holomorphic equivalent to any convex domain since the fundamental group of $D$ is nontrivial. By the same reason, $D$ is not holomorphic equivalent to any Teichm\"uller space since it is well known that all Teichm\"uller spaces are contractible.
\end{proof}
\end{prop}

For general finitely connected planar domains, the boundary behavior of their squeezing functions can be described as follows:

\begin{thm}\label{thm:fin. connected with point}
Let $D$ be a finitely connected planar domain, let $E$ be a connected component of $\bar{\mathbb{C}}- D$: \\
(1)  if $E$ is not a single point, then
$$\lim_{D\ni z\rightarrow E}s_D(z)=1;$$
(2) if $E = \{p\}$ contains a single point, then
 $$s_D(z)\leq \sigma^{-1}(K_{\tilde{D}}(z ,p)),\ \ z\in D,$$
where $\sigma$ is defined as in the proof of Theorem \ref{thm:squeezing function of annuli} and $K_{\tilde{D}}(\cdot , \cdot)$ is the Kobayashi distance on the domain $\tilde{D}:= D\cup\{p\}$. In particular,
 $$\lim_{D\ni z\rightarrow p}s_D(z)=0.$$
\end{thm}
\begin{proof}
The proof of (1) is similar to the proof of Theorem \ref{thm:fin. connected not point} and we will not repeat it here. Now we give the proof of (2).
 Let $z\in D$, and $f_z: D\rightarrow \Delta$ be a holomorphic  embedding such that $f_z(z)=0$. Note that $\tilde{D}$ is a domain. By Riemann's removable singularity theorem, $f_z$ can be extended as a holomorphic map $\tilde{f_z}$ form $\tilde{D}$ to $\Delta$. It is clear that $\tilde{f_z}(p)\not\in f(D)$. By the decreasing property of Kobayashi distances, we have $K_{\Delta}(0 , \tilde{f_z}(p))\leq K_{\tilde{D}}(z ,p)$, hence $s_D(z)\leq \sigma^{-1}(K_{\tilde{D}}(z ,p)).$
\end{proof}

\begin{rem}
Theorem \ref{thm:fin. connected with point} implies that a  finitely connected planar domain is a holomorphic homogeneous regular domain if and only if each connected component of its complement in $\bar{\mathbb{C}}$ is not a single point.
\end{rem}

The  examples of planar holomorphic homogeneous regular domains given by Theorem \ref{thm:fin. connected not point}  are all finitely connected, i.e, their complement in $\bar{\mathbb{C}}$ have finitely many connected components. We can also construct a class of planar holomorphic homogeneous regular domains which are infinitely connected.

Let $Aut(\Delta)$ be the holomorphic automorphism group of $\Delta$, and denote by $\Delta_{r}$  the disc in $\mathbb{C}$ with center $0$ and radius $r$ and $\overline{\Delta}_r$ its closure. We first prove the following lemma:

\begin{lem}\label{lem:useful constants}
For any   positive numbers $u$, $v$ and $w$ with $u<v<w<1$, there exists a positive number $c(u,v,w)$ such that:
for an arbitrary positive number $r$ with $u<r< v$, let $D\subset A_r$ be a domain  containing  $\Delta_{w}\setminus\overline{\Delta}_{r}$, then, for $z\in \Delta_{v}\setminus\overline{\Delta}_{r}$, we have $s_D(z)\geq c(u,v,w).$
\end{lem}
\begin{proof}
Consider the reflection $R: D\rightarrow \Delta$ given by
$$z\mapsto r/z ,$$
It is clear that $z\in R(D)$ provided $r/w < |z|<1$. For $z\in \Delta_{v}\setminus\overline{\Delta}_{r}$, we have
$$r/v<|R(z)|<1,$$
hence  the disc (with respect to the Poincare metric) with center $z$ and radius $\sigma(r/v)-\sigma(r/w)$ is contained in $R(D)$.
So we see that
$$s_{R(D)}(R(z))< \sigma^{-1}(\sigma(r/v)-\sigma(r/w))$$
for $z\in \Delta_{v}\setminus\overline{\Delta}_{r}$. By the biholomorphic invariance of squeezing functions, we get
$$s_D(z)\geq \sigma^{-1}(\sigma(r/v)-\sigma(r/w))$$
for $z\in \Delta_{v}\setminus\overline{\Delta}_{r}$. Take
$$c(u, v, w)= \inf_{u\leq r\leq v}\{\sigma^{-1}(\sigma(r/v)-\sigma(r/w))\},$$
it is clear that $c(u, v, w)>0$ and it satisfies the condition of the lemma.
\end{proof}

By the above lemma, we can prove the following

\begin{thm}\label{thm:infitite connected domain}
Let $u$, $v$ and $w$  be positive numbers with $u<v<w<1$, let $r_k$, $k=1, 2, \cdots$, be a sequence of positive numbers satisfying $u<r_k<v$. Let $f_k$ be a sequence in $Aut(\Delta)$ such that $f_k(\overline{\Delta}_{w})$ are pairwise disjoint. Then the domain
$$D = \Delta\setminus(\cup_{k=1}^{\infty}f_k(\overline{\Delta}_{r_k}))$$
is a holomorphic homogeneous regular domain.
\end{thm}
\begin{proof}
Let $c(u,v,w)$ as in Lemma \ref{lem:useful constants}, denote $c(u, \frac{v+w}{2}, w)$ by $c$.  By  the above lemma and biholomorphic invariance of squeezing functions, we have $s_D(z)\geq c$ for
$$z\in D':= D\cap\left( \cup_{k=1}^{\infty} f_k(\Delta_{\frac{v+w}{2}})\right).$$
For $z\in D\setminus D'$, the distance form $z$ to $\partial D$ with respect to the Poincare distance on $\Delta$ is greater than $\sigma(\frac{v+w}{2})-\sigma (v)$. Take a conformal map $f\in Aut(\Delta)$ such that $f(z)=0$, we see that $s_D(z)\geq \sigma^{-1}(\sigma(\frac{v+w}{2})-\sigma (v))$. So $s_D$ has a positive lower bound and hence $D$ is a holomorphic homogeneous regular domain.
\end{proof}

As a consequence,  the  Carath\'eodory metrics, the Kobayashi metrics, and the Bergman metrics on domains that are constructed in Theorem \ref{thm:infitite connected domain} are complete, and they are equivalent. An explicit example can be constructed as follows: take $f\in Aut(\Delta)$
defined by
$$f(z)= \frac{z+1/2}{1+z/2}$$
and let $D=\Delta\setminus(\cup_{k=-\infty}^\infty f^k(\overline{\Delta}_{1/4}))$, then $D$ is a holomorphic homogeneous regular domain, and the  Carathodory metric, the Kobayashi metric, and the Bergman metric on $D$ are complete and equivalent. The special domain $D$ was constructed in \cite{Krantz} to show that a bounded planar domain may have infinite discrete automorphism group.

\section{Squeezing functions on annuli}\label{sec:annuli}
In the above section, we have studied some properties of squeezing functions of annuli. In this section, we want to investigate further properties of them.

We have seen that all annuli are holomorphic homogeneous regular domains, and their squeezing functions tend to 1 at the boundary.
An interesting but difficult problem is to give an exact expression of $s_{A_r}$. Another relatively simple problem is to determine the minimum of  $s_{A_r}$ for $r>0$, which are conformal invariants.

 By the conformal invariance of $s_{A_r}$, $s_{A_r}(z)$ depends only on $|z|$, so $s_{A_r}$ is reduced to a function defined on $(r,1)$. Be the reflection $z\mapsto r/z$, it can be further reduced to a function on $[\sqrt{r}, 1)$. We show that $s_{A_r}(x)$ is strictly increasing on $[\sqrt{r}, 1)$. To prove this result, we first prove two propositions, which are also interesting in their own right.

 Let $D\subset\mathbb{C}$ be a bounded domain with smooth boundary, denote by $H_D$  the set of univalent maps $f$ form $D$ to $\Delta$ such that $\Delta\setminus f(D)$ is a compact set. For bounded planar domain $D$, we always denote by $P_D(\cdot , \cdot)$ the Poincare distance of $D$.

\begin{prop}\label{prop:extremal implies surjective}
Let $D\subset\mathbb{C}$ be a bounded domain with smooth boundary. Then, for any $p\in D$, we have
$$s_D(p) = \sup\{r| \Delta_r\subset f(D)\ for
 some\ f\in H_D\ with \ f(p)=0 \}.$$
\end{prop}
\begin{proof}
Let $f: D\rightarrow \Delta$ be a univalent map with $f(p)=0$ and assume $\Delta_r\subset f(D)$. Let $D'$ be the union of $f(D)$ and the compact connected components of $\Delta\setminus f(D)$. Then $D'$ is simply connected. By the Riemann mapping theorem, there is a conformal map $g: D'\rightarrow \Delta$ with $g(0)=0$.

Note that
$$\Delta_r=\{z\in \Delta| P_\Delta(z , 0)< \sigma(r) \},$$
by the decreasing property of the Poincare metrics on planar domains, we have $P_\Delta(z,w)\leq P_{D'}(z,w)$ for all $z , w \in D'$, so
$$\{z\in D'| P_{D'}(z , 0)< \sigma(r) \}\subset\Delta_r.$$
Note that $g$ is an isometry form $(D',P_{D'})$ to $(\Delta, P_\Delta)$. Hence
$$\Delta_r=\{z\in \Delta| P_\Delta(z , 0)< \sigma(r) \}= g(\{z\in D'| P_{D'}(z , 0)<\sigma(r)\})\subset g(f(D)).$$
Now we get a univalent map $g\circ f: D\rightarrow \Delta$ with $g\circ f\in H_D$, $g\circ f(p)=0$ and $\Delta_r\subset g\circ f(D)$. Since $f$ is arbitrary, we see that
$$s_D(p) = \sup\{r| \Delta_r\subset h(D)\ for\
 some\ h\in H_D\ with \ h(p)=0 \}.$$
\end{proof}

For two subsets $A$ and $B$ of $\Delta$, we denote by $P_\Delta(A , B)$ the distance between $A$ and $B$ with respect to the Poincare distance of $\Delta$.

\begin{prop}\label{prop:in form of Poincare metric}
Let $D\subset\mathbb{C}$ be a bounded domain with smooth boundary. For $p\in D$. Let
$$u(p) = \sup\{P_\Delta(f(p), \Delta\setminus f(D))| f\in H_D\} ,$$
then we have $s_D(p) = \sigma^{-1}(u(p))$.
\end{prop}
\begin{proof}
For $p\in D$ and $f\in H_D$, let
$$u=P_\Delta(f(p), \Delta\setminus f(D)).$$
Then the $P_\Delta$-disc with center $f(p)$ and radius $u$ is contained in $f(D)$. Take a conformal map $g:\Delta\rightarrow \Delta$ such that $g(f(p))=0$, then we have $\Delta_{\sigma^{-1}(u)}\subset g\circ f(D)$. So we have $s_D(p)\geq \sigma^{-1}(u)$. Since $f\in H_D$ is arbitrary, we have
$s_D(p)\geq \sigma^{-1}(u(p))$.

On the other hand, for an arbitrary $f\in H_D$ with $f(p)=0$, let $r$ be the positive number such that $\Delta_r\in f(D)$ but $\Delta_{r+\epsilon}\nsubseteq f(D)$ for any $\epsilon >0$. It is clear that $$r=\sigma^{-1}\left(P_\Delta(f(p), \Delta\setminus f(D))\right)\leq \sigma^{-1}(u(p)).$$ Since $f\in H_D$ is arbitrary, by the above proposition, we have $s_D(p)\leq \sigma^{-1}(u(p))$.
\end{proof}

\begin{thm}\label{thm:incre. of squ. f. annuli}
Viewed as a function on $[\sqrt{r}, 1)$, the squeezing function $s_{A_r}(z)$ of $A_r$ is strictly increasing on $[\sqrt{r}, 1)$; in particular, it attains its minimum at $\sqrt{r}$.
\end{thm}
\begin{proof}
 For simplicity, let $s=s_{A_r}$. For  $x\in [\sqrt{r}, 1)$, by Theorem\ref{thm:existence of ext}, there is a univalent map $f: D\rightarrow \Delta$ such that $f(x)=0$ and $\Delta_{s(x)}\subset f(D)$. By Proposition\ref{prop:extremal implies surjective}, we may assume $f\in H_D$. By proposition\ref{prop:in form of Poincare metric}, we have $s(x)= \sigma^{-1}(P_\Delta(f(x), \Delta\setminus f(D)))$. By proposition\ref{prop:in form of Poincare metric} and the conformal invariance of $s$, we have the identity
 $$P_\Delta(f(x), \Delta\setminus f(D))=sup\{P_\Delta(f(z), \Delta\setminus f(D))|z\in D, |z|=x\}.$$

 Note that the curve $f(|z|=\sqrt{r})$ separates $\mathbb{C}$ into two connected parts, let $U$ and $V$ be the bounded and unbounded connected components of $\mathbb{C}\setminus f(|z|=\sqrt{r})$ respectively.

 If $x>\sqrt{r}$, then $f(|z|=x)\subset V$. In fact, if it is not the case, then composing $f$ with the reflection $z\mapsto r/z$ will lead to a contradiction to the extremal property assumption on $f$.

 Now let $x'\in [\sqrt{r}, 1)$ with $x'>x$, then it is clear that $f(|z|=x')$ lies in the unbounded component of $\mathbb{C}\setminus f(|z|=x)$,
 so we have
 $$sup\{P_\Delta(f(z), \Delta\setminus f(D))|z\in D, |z|=x'\}>sup\{P_\Delta(f(z), \Delta\setminus f(D))|z\in D, |z|=x\}$$
 by proposition\ref{prop:in form of Poincare metric}, there is a point $z\in A_r$ with $|z|=x'$ and $s_{A_r}(z)>s(x)$. Note that $s_{A_r}(z)=s(|z|)=s(x')$, hence $s(x')>s(x)$.
 \end{proof}

 Theorem \ref{thm:incre. of squ. f. annuli} and its proof lead us to conjecture that, for $\rho\in [\sqrt{r}, 1)$, $s_{A_r}(\rho)$ is given by
$$s_{A_r}(\rho)= \sigma^{-1}(\sigma(\rho)-\sigma(r))=\sigma^{-1} \left(\log\frac{(1+\rho)(1-r)}{(1-\rho)(1+r)}\right),$$
where the function $\sigma$ is defined as in the proof of Theorem \ref{thm:squeezing function of annuli}. Provided  this conjecture,  Theorem \ref{thm:incre. of squ. f. annuli} implies that  $s_{A_r}(\rho)$ attains its minimum             $$s_{A_r}(\sqrt{r})=\tanh\log\frac{1+\sqrt{r}}{\sqrt{1+r}}$$
at $\rho = \sqrt{r}$, which characterizes the conformal structure of $A_r$.

\section{Explicit form of squeezing functions on some special domains}\label{sec:explicit form of sf}

In this section, we give the explicit form of squeezing functions on some special domains, namely punctured balls and classical bounded symmetric domains.

We first consider domains constructed by deleting analytic subsets from other domains.
\begin{thm}\label{thm:punctured bounded domain}
Let $D'\subset\mathbb{C}^n$ be a bounded domain and $A\subset D'$ be a proper analytic subset. Then, for the domain $D=D'\setminus A$, one has
$$s_D(z)\leq \sigma^{-1}(K_{D'}(z ,A)),\ \ z\in D$$
where $\sigma$ is defined as in the proof of Theorem \ref{thm:squeezing function of annuli} and $K_{D'}(\cdot , \cdot)$ is the Kobayashi distance on $D'$; in particular, we have $\lim_{z\rightarrow z_0}s_D(z)=0$.
\end{thm}
\begin{proof}
Let $z\in D$, and $f_z: D\rightarrow B^n$ be a holomorphic open embedding such that $f_z(z)=0$. By Riemann's removable singularity theorem, $f_z$ can be extended to a holomorphic map $\tilde{f_z}$ form $D'$ to $B^n$. It is clear that $\tilde{f_z}(D)\cap \tilde{f_z}(A)=\emptyset$. By the decreasing property of Kobayashi distance, we have $K_{B^n}(0 , \tilde{f_z}(A))\leq K_{D'}(z ,A)$, hence $s_D(z)\leq \sigma^{-1}(K_{D'}(z ,A)),\ \ z\in D$.
\end{proof}

\begin{rem}
The domains $D$ constructed in the above theorem are not holomorphic homogeneous regular domains.
The conclusion  can also be derived from Corollary \ref{cor:holomorphic homogeneous regular domain implies stein} since $D$ is not pseudoconvex, or from Theorem \ref{thm:Caro. metic on holomorphic homogeneous regular domain complete} since the Carath\'{e}odory metric on $D$, which is just the restriction of that of $D'$, is not complete.
\end{rem}

For the special case of punctured balls, Theorem \ref{thm:punctured bounded domain} implies the following
\begin{cor}\label{cor:punctured balls}
The squeezing function $s_{B^n\setminus \{0\}}$ on the $n$ dimensional punctured ball $B^n\setminus \{0\}$ is given by
$$s_{B^n\setminus \{0\}}(z)=\parallel z\parallel$$
where $\parallel z\parallel$ is the Euclidean norm of $z$.
\end{cor}
\begin{proof}
By theorem\ref{thm:punctured bounded domain}, we have $s_{B^n\setminus \{0\}}(z)\leq \parallel z\parallel$. On the other hand,
it is clear that $s_{B^n\setminus \{0\}}(z)\geq \parallel z\parallel$. Hence $s_{B^n\setminus \{0\}}(z)=\parallel z\parallel$.
\end{proof}

Other examples of bounded domains whose squeezing functions can be given explicitly are classical symmetric bounded domains. Recall that a classical symmetric domain is a domain of one of the following four types:
\begin{equation*}
\begin{split}
& D_I(r,s)=\{Z=(z_{jk}): I-Z\bar{Z'}>0,\ \text{where}\ Z\ \text{is\ an}\ r\times s \ \text{matrix} \}\ (r\leq s),\\
& D_{II}(p)=\{Z=(z_{jk}): I-Z\bar{Z'}>0,\ \text{where}\ Z\ \text{is\ a\ symmetric\ matrix\ of\ order\ }p\},\\
&D_{III}(q)=\{Z=(z_{jk}): I-Z\bar{Z'}>0,\ \text{where}\ Z\ \text{is\ a\ skew-symmetric\ matrix\ of\ order\ }q\},\\
&D_{IV}(n)=\{z=(z_1, \cdots , z_n)\in\mathbb{C}^n: 1+|zz'|^2-2zz'>0,\ 1-|zz'|>0\}.
\end{split}
\end{equation*}
Here $I$ is the identity matrix of proper order, $\bar{Z}$ denotes the conjugate matrix of $Z$ and $Z'$ the transposed matrix of $Z$. The complex dimensions of these four domains are $rs,\ p(p+1)/2,\ q(q-1)/2$ and $n$ respectively.

For a bounded homogeneous domain $D$, by the holomorphic invariance of squeezing functions, $s_D$ is a constant function on $D$, and we  denote this constant by $s(D).$  By a theorem of Kubota, which is based on an earlier work of Alexander \cite{Alexander}, the squeezing functions on the above four types of domains can be given explicitly as follows:
\begin{thm}(see Theorem 1 in \cite{Kubota})
\begin{equation*}
\begin{split}
&s(D_I(r,s)) = r^{-1/2},\\
&s(D_{II}(p))=p^{-1/2},\\
&s(D_{III}(q))=[q/2]^{-1/2},\\
&s(D_{IV}(n))= 2^{-1/2},
\end{split}
\end{equation*}
where $[q/2]$ denotes the integral part of $q/2$.
\end{thm}

For products of classical symmetric domains, we have:
\begin{thm}(see Theorem 2 in \cite{Kubota})
If $D_1, \cdots , D_m$ are classical symmetric domains, then
$$s(D_1\times\cdots\times D_m) = (s(D_1)^{-2}+\cdots + s(D_m)^{-2})^{-1/2}.$$
\end{thm}

\begin{rem}In the 1980s, Y. Kubota considered the following Carath\'{e}odory extremal problem \cite{Kubota0,Kubota,Kubota1}: \begin{equation}\label{Kubota extremal problem}M(z_0,D)=\sup_{F\in \mathfrak{F}(D)}\left|J_F(z_0)\right|\qquad(z_0\in D),\end{equation} where $D$ is a bounded domain in the complex Euclidean space $\mathbb{C}^n$ and $\mathfrak{F}(D)$ consists of all holomorphic mappings from $D$ into the unit ball $B^n$ in $\mathbb{C}^n$, and $J_F$ is the Jacobian of $F$.

He proved that the extremal mapping of the extremal problem (\ref{Kubota extremal problem}) is unique up to a unitary linear transformation when $D$ is a bounded symmetric domain(including two exceptional cases) \cite{Kubota1}.
We observe that the extremal mappings are exact the extremal embedding from bounded symmetric domains into the unit ball.
Take $D_I(r,s)$ here for example, we can find from Kubota's proof that the extremal mapping is
$f(z)=z/\sqrt{r},~~ z=(z_{11},\cdots,z_{1s},z_{21},\cdots,z_{rs})\in \mathbb{C}^{rs}$, which is exact an extremal embedding for the squeezing function $s(D_I(r,s))=  r^{-1/2}$ since one knows $B_{n}\subset D_I(r,s)\subset \sqrt{r}B^{n},~~(n=rs)$.

When $D$ is a complex ellipsoid in $\mathbb{C}^n$, i.e.
\[D=D(k_1,\cdots, k_n)=\big\{z\in \mathbb{C}^n:~~\sum_{j=1}^{n}|z_j|^{k_j}<1 \big\},\] where $k_j (j=1,2,\cdots,n)$ are positive real numbers,
Ma considered the extremal problem (\ref{Kubota extremal problem}) in 1997 \cite{Ma}. It is proved that the extremal mapping is again linear, and we conjecture it is likely the extremal embedding for squeezing function $s_{D}(z)$ in this case. Therefore, it will be interesting to consider relations in general between squeezing function $s_D$ on a bounded domain $D$ and the Carath\'{e}odory maps from $D$ into the unit ball, especially when $D$ is homeomorphic to a cell.
\end{rem}

\end{document}